\documentclass[11pt]{article}
\usepackage{amsmath, pb-diagram}
\usepackage{amssymb}
\usepackage{amscd}
\newenvironment{proof}{\par\noindent{\bf Proof \,}}{$\hfill \Box$\par\bigskip}
\date{\empty}
\newtheorem{thm}{Theorem}[section]
\newtheorem{lem}[thm]{Lemma}
\newtheorem{prop}[thm]{Proposition}
\newtheorem{cor}[thm]{Corollary}

\newtheorem{ex}[thm]{Example}

\begin{document}

\title{Strongly $J$-Clean Rings with Involutions}

\author{Huanyin Chen\thanks{ Department of Mathematics, Hangzhou Normal University, Hangzhou,
310036, People's Republic of China, e-mail: huanyinchen@yahoo.cn} ,
Abdullah Harmanc\i\thanks{ Hacettepe University, Department of
Mathematics,
 06800 Beytepe Ankara, Turkey, e-mail: abdullahharmanci@gmail.com} , A. \c Ci\u gdem \" Ozcan\thanks{Corresponding author.  Hacettepe University, Department of Mathematics,
 06800 Beytepe Ankara, Turkey, e-mail: ozcan@hacettepe.edu.tr}}

\maketitle

\begin{abstract} \noindent A ring with an involution $*$ is called strongly $J$-$*$-clean if every element is a sum
of a projection and an element of the Jacobson radical that
commute. In this article, we prove several results characterizing
this class of rings. It is shown that a $*$-ring $R$ is strongly
$J$-$*$-clean, if and only if $R$ is uniquely clean and strongly
$*$-clean, if and only if $R$ is uniquely strongly $*$-clean, that
is, for any $a\in R$, there exists a unique projection $e\in R$
such that $a-e$ is invertible and $ae=ea$. 
\end{abstract}

\section{Introduction}

A ring $R$ is {\em strongly clean} provided that every element is
a sum of an idempotent and a unit  that commute with each other
(\cite{N}). A ring $R$ is {\em strongly $J$-clean} provided that
every element is a sum of an idempotent and an element in its
Jacobson radical that commute (\cite{Che1}). A ring $R$ is {\em
uniquely strongly clean} provided that for any $a\in R$, there
exists a unique idempotent $e\in R$ such that $a-e$ is invertible
and $ae=ea$ (\cite{CWZ}). Note that $\{$uniquely strongly clean
rings$ \}\subsetneqq \{$ strongly $J$-clean rings$\}\subsetneqq
\{$ strongly clean rings$\}$. The above mentioned classes of rings
have been getting much attention. Their relations with other
classes of rings, such as unit-regular rings, strongly
$\pi$-regular rings and others, have been studied in the past.

An {\em involution} of a ring $R$ is an anti-automorphism whose square
is the identity map. A ring $R$ with involution $*$ is
called a {\em $*$-ring}. All $C^*$-algebras and Rickart $*$-rings are $*$-rings.
Furthermore, every commutative ring can be regarded as a $*$-ring
with the identity involution $*$. An element $e$ in a $*$-ring R
is called a {\em projection} if $e^2 = e = e^*$. A $*$-ring $R$ is
{\em strongly $*$-clean} if each of its elements is a sum of a unit and
a projection that commute with each other (see \cite{LZ, Va}).

The main purpose of this paper is to explore strong $J$-cleanness
for $*$-rings.  We call a $*$-ring $R$ {\em strongly
$J$-$*$-clean} provided that every element is a sum of a
projection and an element in its Jacobson radical  that commute
with each other. We show several results characterizing this class
of rings. It is proved that a $*$-ring $R$ is strongly
$J$-$*$-clean, if and only if $R$ is uniquely clean and strongly
$*$-clean (Theorem~\ref{corJ}). We say that a $*$-ring $R$ is {\em
uniquely strongly $*$-clean} provided that for any $a\in R$, there
exists a unique projection $e\in R$ such that $a-e$ is invertible
and $ae=ea$. We show that strongly $J$-$*$-clean rings and
uniquely strongly $*$-clean rings are equivalent notions unlike
their non-involutive counterparts (Theorem~\ref{corJ}). Also it is
proved that $R$ is strongly $J$-$*$-clean if and only if $R$ is
strongly $*$-clean and the Jacobson radical $J(R)=\{ x\in R~|~1-x
$\, is invertible $ \}$ (Theorem~\ref{str34}). As consequences,
various properties of strongly $J$-$*$-clean rings are derived.

Throughout, all rings are associative with identity. We use $U(R)$
to denote the set of all invertible elements in the ring $R$ and
$J(R)$ always stands for the Jacobson radical of $R$.

\section{Strongly $J$-$*$-Clean Rings}

The aim of this section is to characterize strongly $J$-$*$-clean
rings by means of strong $*$-cleanness. A $*$-ring $R$ is
$J$-$*$-clean in the case that every element is a sum of an
idempotent and an element in this Jacobson radical. We begin with
the following result.

\begin{prop}\label{prop1} Let $R$ be a $*$-ring. Then the following are equivalent:
\begin{itemize}
\item[\rm{(1)}] $R$ is strongly $J$-$*$-clean.
\item[\rm{(2)}] $R$ is strongly $J$-clean and strongly $*$-clean.
\item[\rm{(3)}] $R$ is abelian and $R$ is $J$-$*$-clean.
\end{itemize}
\end{prop}
\begin{proof} (1) $\Rightarrow$ (2) Clearly, $R$ is strongly $J$-clean. For any $a\in R$,
there exists a projection $e\in R$ such that $u:=a-e\in J(R)$ and
$ae=ea$. Then $a=(1-e)+(2e-1+u)$. Obviously, $1-e\in R$ is a
projection. As $(2e-1)^2=1$, we see that $2e-1+u\in U(R)$. Thus,
$R$ is strongly $*$-clean.

$(2)\Rightarrow (3)$ Since $R$ is strongly $*$-clean, it follows
from \cite[Theorem 2.2]{LZ} that $R$ is abelian and every
idempotent is a projection. Therefore $R$ is $J$-$*$-clean.

$(3)\Rightarrow (1)$ is obvious.
\end{proof}

\begin{ex}\label{ilk}  {\rm (1)} {\em Let the ring $R = \mathbb{Z}_2\oplus \mathbb{Z}_2$. Define $\sigma : R\rightarrow R$ by
$\sigma(x, y) = (y, x)$. Consider the ring $T_2(R, \sigma) =
\big\{\left (\begin{array}{cc}a & b\\0 & a\end{array}\right)\mid a,
b\in R \big\}$ with the following operations:

\begin{center}$\left (\begin{array}{cc}a & b\\0 &
a\end{array}\right) + \left (\begin{array}{cc}c & d\\0 &
c\end{array}\right) = \left (\begin{array}{cc}a+c & b+d\\0 &
a+c\end{array}\right)$,\,\, $\left (\begin{array}{cc}a & b\\0 &
a\end{array}\right).\left (\begin{array}{cc}c & d\\0 &
c\end{array}\right)=\left (
\begin{array}{cc}ac & ad+b\sigma(c)\\0 &
ac\end{array}\right).$\end{center}
 Then $J(T_2(R, \sigma)) =
\big\{\left (\begin{array}{cc}0 & b\\0 & 0\end{array}\right)\mid
b\in R \big\}$ is nilpotent, and $T_2(R, \sigma)/J(T_2(R,
\sigma))\cong R$ is Boolean. Define $* : R\rightarrow R$ by $\left
(\begin{array}{cc}a & b\\0 & a\end{array}\right)^*=\left
(\begin{array}{cc}a & \sigma(b)\\0 & a\end{array}\right)$. As
$\sigma^2 = 1_R$, we easily check that $*$ is an involution of
$T_2(R,\sigma)$. For any $a, b\in R$, $\left (\begin{array}{cc}a &
0\\0 & a\end{array}\right)$ is a projection. Furthermore, $\left
(\begin{array}{cc}a & b\\0 & a\end{array}\right)-\left
(\begin{array}{cc}a & 0\\0 & a\end{array}\right)\in
J(T_2(R,\sigma))$. Therefore, $T_2(R,\sigma)$ is a $J$-$*$-clean
ring. But it is not strongly $J$-$*$-clean, as $\left
(\begin{array}{cc}(0,1) & (0,0)\\(0,0)
& (0,1)\end{array}\right)$ is not central.}\\

{\rm (2)} {\em Let $R=\mathbb{Z}_{(3)}$ be the localization, and $*=1_R$, the
identical automorphism of $R$. Since $R$ is local, it is strongly
$*$-clean.  Since $R/J(R)$ is not
Boolean, for example $(\frac{2}{1})^2-\frac{2}{1}\not\in
J(R)=3R$, $R$ is not strongly $J$-$*$-clean (see Proposition~\ref{prop26}).} \\

{\rm (3)} {\em Let $R=\left( \begin{array}{ll}
\mathbb{Z}_2& \mathbb{Z}_2\\
0&\mathbb{Z}_2
\end{array}
\right)$. By \cite[Example 2]{N} $R$ is strongly clean, and it is clear that $R/J(R)$ is Boolean. Then $R$ is strongly $J$-clean by \cite[Theorem 2.3]{Che1}. But $R$ cannot be
strongly $J$-$*$-clean for any involution $*$ on $R$ because it is not abelian.}
\end{ex}

Let $R$ be a $*$-ring, and let $C(R)=\{ a\in R~|~ax=xa$ for any
$x\in R\}$. Then $C(R)$ is a subring of the ring $R$. It is easy
to check that $*: C(R)\to C(R)$ is also an anti-automorphism.
Thus, $C(R)$ is a $*$-ring.

\begin{cor}\label{denkk} Let $R$ be a $*$-ring. Then $R$ is strongly $J$-$*$-clean if and only if
\begin{itemize}
\item[{\rm(1)}] $C(R)$ is $J$-$*$-clean;
\item[{\rm(2)}] $R=C(R)+J(R)$.
\end{itemize}
\end{cor}
\begin{proof} Suppose that $R$ is strongly $J$-$*$-clean. Then $R$
is strongly $*$-clean by Proposition~\ref{prop1}. Thus, every
projection in $R$ is central from~\cite[Theorem 2.2]{LZ}. For any
$a\in C(R)$, there exist a projection $e\in R$ and an element
$u\in J(R)$ such that $a=e+u$. Hence, $e\in C(R)$ is a projection.
This implies that $u\in J(R)\cap C(R)\subseteq J\big(C(R)\big)$.
As a result, $C(R)$ is $J$-$*$-clean. On the other hand, it
follows from Proposition~\ref{prop1} that $R$ is abelian strongly
$J$-clean. In view of \cite[Corollary 16.4.16]{Ch}, $R$ uniquely
clean. According to \cite[Proposition 25]{NZ}, $R=C(R)+J(R)$.

Conversely, assume that $(1)$ and $(2)$ hold. For any idempotent
$e\in R$, there exist some $b\in C(R)$ and $c\in J(R)$ such that
$e=b+c$. As $C(R)$ is $J$-$*$-clean, we can find a projection
$f\in C(R)$ and a $w\in J\big(C(R)\big)$ such that $b=f+w$. Hence,
$e=f+(w+c)$, and so $e-f=w+c$. Obviously,
$(e-f)\big(1-(e-f)^2\big)=0$. On the other hand, $(e-f)^2\in
J\big(C(R)\big)+J(R)$, and so $1-(e-f)^2\in U(R)$. Therefore
$e=f$, i.e., every idempotent in $R$ is a projection. For any
$a\in R$, write $a=s+t, s\in C(R), t\in J(R)$. We have a
projection $g\in R$ such that $s=g+v$ with $v\in J\big(C(R)\big)$.
This shows that $a=g+(t+v)$, and so $R$ is $J$-$*$-clean.
According to \cite[Lemma 2.1]{LZ} and Proposition~ \ref{prop1}, we complete the
proof.\end{proof}

According to \cite[Theorem 2.1]{Che2}, a ring $R$ is uniquely
clean if and only if $R$ is an abelian exchange ring and $R/M\cong
\mathbb{Z}_2$ for all maximal ideals $M$ of $R$. For strongly
$J$-$*$-clean rings, we derive the following result.

\begin{prop}\label{str7} Let $R$ be a $*$-ring. Then $R$ is strongly $J$-$*$-clean if and only if
\begin{enumerate}
\item[\rm{(1)}] $R$ is strongly $*$-clean ring;
\item[\rm{(2)}] $R/M\cong {\Bbb Z}_2$ for all maximal ideals $M$ of $R$.
\end{enumerate}
\end{prop}
\begin{proof} Suppose that $R$ is strongly $J$-$*$-clean. Then $R$ is an abelian strongly $J$-clean ring by
Proposition~\ref{prop1}. Thus, $R$ is uniquely clean from
\cite[Corollary 16.4.16]{Ch}. Therefore $R/M\cong {\Bbb Z}_2$ for
all maximal ideals $M$ of $R$ by \cite[Theorem 2.1]{Che2}.

Conversely, assume that $(1)$ and $(2)$ hold. Then every
idempotent in $R$ is a projection. Further, $R$ is an abelian
exchange ring. According to \cite[Theorem 2.1]{Che2}, $R$ is
uniquely clean, and so $R$ is strongly $J$-clean. Accordingly,
$R$ is strongly $J$-$*$-clean by Proposition~\ref{prop1}.
\end{proof}

\begin{cor}\label{str13} Let $R$ be a $*$-ring. Then $R$ is strongly $J$-$*$-clean if and only if
\begin{enumerate}
\item[\rm{(1)}] $R$ is strongly $*$-clean ring;
\item[\rm{(2)}] For all maximal ideal $M$ of $R$, $1$ is not the sum of two units in $R/M$.
\end{enumerate}
\end{cor}
\begin{proof} One direction is obvious by the previous result. Conversely, assume that $(1)$ and $(2)$
hold. Let $M$ be a maximal ideal of $R$. Clearly, $R/M$ is an
abelian exchange ring; hence, $R/M$ is an exchange ring with
artinian primitive factors. As $R/M$ is simple, $J(R/M)=0$. Let
$f\in R/M$ be an idempotent. Then $(R/M)f(R/M)=0$ or $R/M$, and so
$f=\overline{0}$ or $\overline{1}$. This means that $R/M$ is
indecomposable. Therefore $R/M$ is simple artinian. Thus,
$R/M\cong M_n(D)$, where $D$ is a division ring. As every
idempotent in $R/M$ is central, we deduce that $n=1$. This implies
that $R/M\cong D$. If $|D|\geq 3$, then we can find a set $\{ 0,
1,x\}\subseteq D$, where $x\neq 0,1$. This shows that $1-x\in
U(R/M)$, a contradiction. Hence, $R/M\cong {\Bbb Z}_2$. According
to Proposition~\ref{str7}, we complete the proof.\end{proof}

\begin{prop}\label{prop26} Let R be a $*$-ring. Then the following are equivalent:
\begin{enumerate}
\item[\rm{(1)}] $R$ is strongly $J$-$*$-clean.
\item[\rm{(2)}] $R/J(R)$ is Boolean and $R$ is strongly $*$-clean.
\end{enumerate}
\end{prop}
\begin{proof}  (1) $\Rightarrow$ (2) In view of
Proposition~\ref{prop1}, $R$ is strongly $*$-clean and $R$ is
strongly $J$-clean. By virtue of~ \cite[Proposition 16.4.15]{Ch},
$R/J(R)$ is Boolean.

(2) $\Rightarrow$ (1) As $R$ is strongly $*$-clean, it is strongly
clean. According to \cite[Proposition 16.4.15]{Ch}, $R$ is
strongly $J$-clean. In light of Proposition~\ref{prop1}, $R$ is a
strongly $J$-$*$-clean ring.\end{proof}

\begin{cor}\label{pac9} Let R be a $*$-ring. Then $R$ is strongly $J$-$*$-clean if and only if
\begin{enumerate}
\item[\rm{(1)}] $R$ is strongly $*$-clean;
\item[\rm{(2)}] Every nonzero idempotent in $R$ is not the sum of two units.
\end{enumerate}
\end{cor}
\begin{proof} Suppose that $R$ is a strongly $J$-$*$-clean ring.
By virtue of Proposition~\ref{prop26}, $R$ is a strongly $*$-clean
ring, and $R/J(R)$ is Boolean. Let $0\neq e\in R$ be an
idempotent. If $e=u+v$ for some $u,v\in U(R)$. Then
$\overline{e}=\overline{u}+\overline{v}$ in $R/J(R)$. As
$\overline{u}=\overline{v}=\overline{1}$, we see that $2-e\in
J(R)$. In light of~ \cite[Proposition 3.1]{Che1}, $2\in J(R)$;
hence, $e\in J(R)$. This implies that $e=0$, a contradiction.
Therefore every nonzero idempotent in $R$ is not the sum of two
units.

Conversely, assume that $(1)$ and $(2)$ hold. Then $R$ is an
exchange ring. According to \cite[Theorem 13]{LZ1}, $R/J(R)$ is
Boolean. Therefore we complete the proof by
Proposition~\ref{prop26}.
\end{proof}

Recall that a ring $R$ is {\em local} if $R$ has only one maximal right
ideal. As is well known, a ring $R$ is local if and only if
$a+b=1$ in $R$ implies that either $a$ or $b$ is invertible.

\begin{cor} Let $R$ be a local $*$-ring. Then the following are equivalent:
\begin{enumerate}
\item[\rm{(1)}] $R$ is strongly $J$-$*$-clean.
\item[\rm{(2)}] $R$ is strongly $J$-clean.
\item[\rm{(3)}] $R$ is uniquely clean.
\item[\rm{(4)}] $R/J(R)\cong \mathbb{Z}_2$.
\item[\rm{(5)}] $1$ is not the sum of two units in $R$.
\end{enumerate}
\end{cor}
\begin{proof} $(1)\Rightarrow (2)$ is trivial.

$(2),(3)$ and $(4)$ are equivalent by \cite[Lemma 4.2]{Che1}.

$(4)\Rightarrow (5)$ is obvious.

$(5)\Rightarrow (1)$ Since $R$ is a local $*$-ring, $R$ is
strongly $*$-clean. Therefore the result follows from Corollary
~\ref{pac9}.
\end{proof}

A $*$-ring R is called {\it $*$-regular} if $R$ is (von Neumann)
regular and the involution is proper, equivalently for every $x$
in $R$ there exists a projection $p$ such that $xR = pR$ (see
\cite{Be}).

\begin{cor}\label{pac3}  Let $R$ be a $*$-regular ring. Then $R$ is strongly $J$-$*$-clean if and only if $R$ is Boolean.
\end{cor}
\begin{proof} Suppose that $R$ is Boolean. Then $R/J(R)$ is
Boolean. For any idempotent $e\in R$, there exists a projection
$p$ such that $eR=pR$. As $R$ has stable range one, we have a unit
$u\in R$ such that $e=pu$. Clearly, $u=1$, and so $e=p$. This
implies that $R$ is strongly $*$-clean. According to
Proposition~\ref{prop26}, $R$ is strongly $J$-$*$-clean.

Conversely, assume that $R$ is strongly $J$-$*$-clean. Then
$R/J(R)$ is Boolean by Proposition~\ref{prop26}. But $R$ is
regular, and so $J(R)=0$. Therefore $R$ is Boolean.\end{proof}

\section{Uniqueness of Projections}

We start this section by studying the relationship between strong
$J$-$*$-cleanness and uniqueness which will be repeatedly used in
the sequel.

\begin{lem}\label{str3} Let $R$ be a $*$-ring. Then
the following are equivalent:
\begin{enumerate}
\item[\rm{(1)}] $R$ is strongly $J$-$*$-clean.
\item[\rm{(2)}] $R$ is uniquely clean and for any $a\in R$, $a-a^*\in
J(R)$.
\item[\rm{(3)}] $R$ is uniquely clean and for any $a\in R$, $a+a^*\in J(R)$.
\end{enumerate}
\end{lem}
\begin{proof} $(1)\Rightarrow
(2)$ In view of Proposition ~\ref{prop1}, $R$ is an abelian
strongly $J$-clean ring, whence it is uniquely
clean~\cite[Corollary 3.4]{Che2}. For any $a\in R$, there exist a
projection $e\in R$ and an element $u\in J(R)$ such that
$a=e+u,ae=ea$. Thus, $a^*=e^*+u^*$. As $\big(J(R)\big)^*\subseteq
J(R)$, we see that $a-a^*=(e-e^*)+(u-u^*)\in J(R)$.

$(2)\Rightarrow (3)$ Since $R$ is uniquely clean, it follows from
\cite[Lemma 18]{NZ}, $2\in J(R)$. Therefore $a+a^*=(a-a^*)+2a^*\in
J(R)$, as desired.

$(3)\Rightarrow (1)$ Since $R$ is uniquely clean, $R$
is abelian strongly $J$-clean (see \cite[Corollary 16.4.16]{Ch}). For any idempotent $e\in R$,
$e+e^*\in J(R)$. Thus, $(e-e^*)(e+e^*)^2=e-e^*$; hence,
$(e-e^*)\big(1-(e+e^*)^2\big)=0$. This implies that $e=e^*$, i.e.,
every idempotent is a projection. Hence, $R$ is strongly
$J$-$*$-clean. \end{proof}

By the referee's suggestion, we say that a $*$-ring is {\em uniquely
$J$-$*$-clean} provided that for any $a\in R$, there exists a
unique projection $e\in R$ such that $a-e\in J(R)$. We come now to
prove the following main result.

\begin{thm}\label{corJ} Let $R$ be a $*$-ring. Then the following are equivalent:
\begin{enumerate}
\item[\rm{(1)}] $R$ is strongly $J$-$*$-clean.
\item[\rm{(2)}] $R$ is uniquely clean and $R$ is strongly
$*$-clean.
\item[\rm{(3)}] $R$ is uniquely strongly
$*$-clean.
\item[\rm{(4)}] $R$ is uniquely $J$-$*$-clean.
\item[\rm{(5)}] For any $a\in R$, there exists a unique idempotent $e\in R$
such that $a-e\in U(R), ae=ea, ae^*=e^*a$ and $e-e^*\in J(R)$.
\end{enumerate}
\end{thm}
\begin{proof} $(1)\Rightarrow (2)$ is clear from Lemma~\ref{str3} and Proposition~\ref{prop1}.

$(2)\Rightarrow (3)$ For any $a\in R$, it follows from
\cite[Theorem 20]{NZ} that there exists a unique idempotent $e\in
R$ such that $a-e\in J(R)$. Since $R$ is strongly $*$-clean, $e$
is a central projection. Therefore there exists a unique
projection $e\in R$ such that $a-e\in U(R)$ and $ae=ea$, as
required.

$(3)\Rightarrow (1)$ Clearly, $R$ is strongly $*$-clean. In light
of \cite[Theorem 2.2]{LZ}, $R$ is abelian and every idempotent in
$R$ is a projection. Thus, $R$ is uniquely clean. Let $a\in R$.
According to ~\cite[Theorem 20]{NZ}, there exists an idempotent
$e\in R$ such that $a-e\in J(R)$. Thus, we have a projection $e\in
R$ such that $a-e\in J(R)$ and $ae=ea$. Therefore $R$ is strongly
$J$-$*$-clean.

$(1)\Rightarrow (4)$ is trivial by Lemma~\ref{str3}.

$(4)\Rightarrow (1)$ For any idempotent $e$ in $R$, there exists a projection $g\in R$ such that $e-g\in J(R)$. Thus, $e^*-g\in J(R)$.
Therefore $e-e^*=(e-g)-(e^*-g)\in J(R)$. Clearly, there exists a projection $h\in R$ such that $2=h+w$ where $w\in J(R)$. Thus, $1-h=-1+w\in U(R)$.
Hence $1-h=1$, and so $h=0$. This implies that $2\in J(R)$. As a result,
we deduce that $e+e^*=(e-e^*)+2e^*\in J(R)$.

Set $z=1+(e-e^*)^*(e-e^*)$. Write $t=z^{-1}$. Since $z^*=z$,
$t^*=t$. Also $e^*z=e^*ee^*=ze^*$, and so $e^*t=te^*$, and
$et=te$. Set $f=e^*et=te^*e$. Then $$f^*=f,
f^2=e^*ete^*et=e^*ee^*(tet)=e^*ztet=e^*et=f,
fe=f~\mbox{and}~ef=ee^*et=ezt=e.$$ Now $e=f+(e-f)$ and
$e-f=e-e^*et=ee^* e t-e^*et=(e-e^*)e^*et\in J(R)$. Here
$f=f^*=f^2$. In addition, $f=e^*e[1+(e^*-e)(e-e^*)]^{-1}$.

Set $z^\prime=1+(e^*-e)^*(e^*-e)$. Write
$t^\prime=(z^\prime)^{-1}$. Since $(z^\prime)^*=z^\prime$,
$(t^\prime)^*=t^\prime$. Also $e z^\prime=ee^*e=z^\prime e$. Set
$f^\prime=ee^*t^\prime=t^\prime e e^*$. As in the preceding proof,
we see that $f^\prime=(f^\prime)^2=(f^\prime)^*$ and
$ef^\prime=f^\prime$, $f^\prime e=e$. In addition,
$$e-f^\prime=f^\prime e-f^\prime=t^\prime e e^*(e-e^*)\in J(R),$$
where $f^\prime=[1+(e-e^*)(e^*-e)]^{-1} ee^*$.

Thus we get $e=f+(e-f)=f^\prime+(e-f^\prime)$ with $e-f,
e-f^\prime \in J(R)$, $f$ and $f^\prime$ are projections. By the
uniqueness, we get $$e^*e\big[ 1+(e^*-e)(e-e^*)\big]^{-1}= \big[
1+(e-e^*)(e^*-e)\big]^{-1}ee^*.$$ This shows that
$$[1+(e-e^*)(e^*-e)]e^*e = ee^*[1+(e^*-e)(e-e^*)].$$ Obviously,
$(e-e^*)(e^*-e)e^*e=-e^*e+e^*ee^*e$ and
$ee^*(e^*-e)(e-e^*)=-ee^*+ee^*ee^*$. Consequently,
$e^*ee^*e=ee^*ee^*$. One easily checks that
$$\begin{array}{c}
(e-e^*)^3-(e-e^*)=-ee^*e+e^*ee^*;\\
\big[ (e-e^*)^3-(e-e^*)\big] (e+e^*)=(e-e^*)^3-(e-e^*).
\end{array}$$ Thus
$(e-e^*)((e-e^*)^2-1)((e+e^*)-1)=0$. As $e-e^*$, $e+e^*\in J(R)$,
we see that $(e-e^*)^2-1, (e+e^*)-1\in U(R)$, and so $e=e^*$.
Therefore every idempotent is a projection,  and so $R$ is abelian
by \cite[Lemma 2.1]{LZ}, hence the result follows.

$(2)+(3) \Rightarrow (5)$ By (3), there exists
an idempotent $f\in R$ such that $a-f\in U(R), af=fa, af^*=f^*a$
and $f-f^*=0\in J(R)$. The uniqueness of
such idempotent immediately follows from the uniquely cleanness of $R$.

$(5)\Rightarrow (3)$  Let $a\in R$. Then there exists an idempotent $e\in R$ such that $u:=a-e\in
U(R), ae=ea, ae^*=e^*a$ and $e-e^*\in J(R)$. Let
$p=1+(e^*-e)^*(e^*-e)$. As $ae=ea, ae^*=e^*a$, we see that
$ap=pa$. Clearly, $p\in U(R)$. Write $q=p^{-1}$. Then $p^*=p$, this implies $q^*=q$. Further,
$$ep=e(1-e-e^*+ee^*+e^*e)=ee^*e=(1-e-e^*+ee^*+e^*e)e=pe.$$
Thus, we
see that $eq=qe$ and $e^*q=qe^*$. Set $g=ee^*q$. Then
$$g^2=ee^*qee^*q=qee^*ee^*q=qpee^*q=ee^*q=g.$$
In addition,
$g^*=q^*ee^*=ee^*q=g$, i.e., $g\in R$ is a projection. As $aq=qa$,
we see that $ag=ga$. One easy checks that $eg=g$ and
$ge=ee^*qe=ee^*eq=epq=e$. This implies that
$$e-g=e-ee^*p^{-1}=e\big(e(e^*-e)^*-1\big)(e^*-e)p^{-1}\in J(R),$$
and so $e-g+u\in U(R)$. Thus we have a projection $g\in R$ such that $a-g\in U(R)$ and $ga=ag$. The uniqueness in (5) completes the proof.
\end{proof}

A ring element is said to be {\em clean} if it is a sum of a unit and an idempotent. Other related concepts can be defined for ring elements analogously. With such definitions in mind, let us note that the conditions of Theorem~\ref{corJ} are not equivalent when referring to a single element. The following example displays a ring with an element that is uniquely strongly $*$-clean and not uniquely clean. Thus, conditions (2) and (3) are not equivalent when applied to a single element.

 Let $\left( \begin{array}{cc} 1&0\\
0&0
\end{array}
\right)\in M_2({\Bbb Z})$, where $M_2({\Bbb Z})$ is a $*$-ring
with the involution $*: A\mapsto A^T$. One easily checks that
there exists a unique projection $\left( \begin{array}{cc} 0&0\\
0&1
\end{array}
\right)\in M_2({\Bbb Z})$ such that $\left( \begin{array}{cc} 1&0\\
0&0
\end{array}
\right)-\left( \begin{array}{cc} 0&0\\
0&1
\end{array}
\right)=\left( \begin{array}{cc} 1&0\\
0&-1
\end{array}
\right)\in M_2({\Bbb Z})$ is invertible. But we have an
idempotent $\left( \begin{array}{cc} 0&1\\
0&1
\end{array}
\right)\in M_2({\Bbb Z})$ such that $\left( \begin{array}{cc} 1&0\\
0&0
\end{array}
\right)-\left( \begin{array}{cc} 0&1\\
0&1
\end{array}
\right)=\left( \begin{array}{cc} 1&-1\\
0&-1
\end{array}
\right).$\\



We note that the unique projection in (3) or (4) of
Theorem~\ref{corJ} can not be replaced by the unique idempotent
even for a commutative $*$-ring as the following example shows:

\begin{ex} \rm Let $R=\big \{\left(\begin{array}{ll}
 0& 0\\
 0& 0
\end{array} \right), \left(\begin{array}{ll}
 1& 0\\
 0& 1
\end{array} \right),\left(\begin{array}{ll}
 1& 1\\
 0& 0
\end{array} \right),\left(\begin{array}{ll}
 0& 1\\
 0& 1
\end{array} \right)\big\}$ where $0,1\in \mathbb{Z}_2$.  Define $*: R\rightarrow R$, $\left(\begin{array}{ll}
 a& b\\
 c& d
\end{array} \right) \mapsto \left(\begin{array}{cc}
 a+b& b\\
 a+b+c+d& b+d
\end{array} \right)$. Then $R$ is a commutative $*$-ring with the usual matrix addition and multiplication. In fact, $R$ is Boolean, and so, for any $a\in R$, there exists a unique idempotent $e\in R$ such that $a-e\in U(R)$ (or, $a-e\in J(R)$) and $ae=ea$. But $R$ is not strongly $J$-$*$-clean, even not a $*$-clean ring.
\end{ex}

\begin{thm}\label{str34} Let $R$ be a $*$-ring. Then $R$ is strongly $J$-$*$-clean if and only if
\begin{enumerate}
\item[\rm{(1)}] $R$ is strongly $*$-clean;
\item[\rm{(2)}] $J(R)=\{ x\in R~|~1-x\in U(R)\}$.
\end{enumerate}
\end{thm}
\begin{proof} Suppose that $R$ is strongly
$J$-$*$-clean. By virtue of Lemma ~\ref{str3}, $R$ is uniquely
clean. Obviously, $J(R)\subseteq \{ x\in R~|~1-x\in U(R)\}.$
Suppose that $1-x\in U(R)$. If $x\not\in J(R)$, then $0\neq
xR\nsubseteq J(R)$. In view of \cite[Lemma 17]{NZ}, there exists
an idempotent $0\neq e\in xR$. Write $e=xr$ for an $r\in R$. Then
$e=(exe)(ere)$ as every idempotent in $R$ is central. It is easy
to see that $R$ is directly finite. Thus, $exe\in U(eRe)$. In view
of \cite[Corollary 5]{NZ}, $eRe$ is uniquely clean. Clearly,
$0+exe=e+e(x-1)e$. The uniqueness implies that $0=e$, a
contradiction. Therefore $x\in J(R)$, and so $\{ x\in R~|~x-1\in
U(R)\}\subseteq J(R)$. Thus, $J(R)=\{ x\in R~|~1-x\in U(R)\}$.

Conversely, assume that $(1)$ and $(2)$ hold. Let $a\in R$. Then
we can find a projection $e\in R$ such that $(a-1)-e\in U(R)$ and
$e(a-1)=(a-1)e$. That is, $(1-a)+e\in U(R)$. As $1-(a-e)\in U(R)$,
by hypothesis, $a-e\in J(R)$. In addition, $ea=ae$. Therefore $R$
is strongly $J$-$*$-clean.\end{proof}

\begin{cor}\label{str8} Let $R$ be a $*$-ring. Then $R$ is strongly $J$-$*$-clean if and only if
\begin{enumerate}
\item[\rm{(1)}] $R$ is strongly $*$-clean;
\item[\rm{(2)}] For any $a\in R$, $a+a^*\in J(R)$;
\item[\rm{(3)}] $J(R)=\{ x\in R~|~1+xx^*\in U(R)\}$.
\end{enumerate}
\end{cor}
\begin{proof} Suppose that $R$ is strongly $J$-$*$-clean. By virtue of Proposition~\ref{prop1} and Lemma ~\ref{str3},
$R$ is a strongly $*$-clean ring with $2\in J(R)$. For any $x\in
R$, it follows from Lemma ~\ref{str3} that $x+x^*\in J(R)$. If
$1+xx^*\in U(R)$, then $(1+x)(1+x^*)=1+x+x^*+xx^*\in U(R)$. This
implies that $1+x\in R$ is right invertible. Similarly, $1+x\in R$
is left invertible. In view of Theorem~\ref{str34}, $-x\in J(R)$,
and so $x\in J(R)$. This shows that $J(R)=\{ x\in R~|~1+xx^*\in
U(R)\}$.

Conversely, assume that $(1), (2)$ and $(3)$ hold. If $1+x\in
U(R)$, then $1+x^*\in U(R)$; hence, $(1+x)(1+x^*)\in U(R)$. As
$x+x^*\in J(R)$, we see that
$1+xx^*=(1+x)(1+x^*)-(x+x^*)=(1+x)(1+x^*)\big(1-((1+x)(1+x^*))^{-1}(x+x^*)\big)\in
U(R)$. This implies that $x\in J(R)$, and so $-x\in J(R)$. As a
result, $J(R)=\{ x\in R~|~1-x\in U(R)\}$. According to
Theorem~\ref{str34}, the result follows.\end{proof}

Recall that a group $G$ with an identity $e$ is {\em torsion} provided
that for any $g\in G$ there exists some $n\in {\Bbb N}$ such that
$g^n=e$.

\begin{cor}\label{str12} Let $R$ be a $*$-ring. Then $R$ is strongly $J$-$*$-clean
if and only if
\begin{enumerate}
\item[\rm{(1)}] $R$ is strongly $*$-clean;
\item[\rm{(2)}] $2\in J(R)$;
\item[\rm{(3)}] $U\big(R/J(R)\big)$ is torsion.
\end{enumerate}
\end{cor}
\begin{proof} If $R$ is strongly $J$-$*$-clean, then $R$ is strongly
$*$-clean and $2\in J(R)$. In addition, $R/J(R)$ is Boolean. Thus,
$U\big(R/J(R)\big)=\{ \overline{1}\}$ is torsion.

Conversely, assume that $(1)$, $(2)$ and $(3)$ hold. Assume that
$1-x\in U(R)$. Then $\overline{1-x}\in U\big(R/J(R)\big)$. By
hypothesis, there exists some $n\in {\Bbb N}$ such that
$(\overline{1-x})^n=\overline{1}$, and so
$(\overline{1-x})^{2n}=\overline{1}$. As $2\in J(R)$, we see that
$x^{2n}\in J(R)$. Clearly, $R$ is an abelian exchange ring, and
then so is $R/J(R)$. This implies that $R/J(R)$ is reduced, i.e.
it has no nonzero nilpotent elements, and so $x\in J(R)$. This
implies that $J(R)=\{ x\in R~|~1-x\in U(R)\}$. Accordingly, $R$ is
strongly $J$-$*$-clean by Theorem~\ref{str34}.\end{proof}

We say that an ideal $I$ of a $*$-ring $R$ is a {\em $*$-ideal}
provided that $I^*\subseteq I$. If $I$ is a $*$-ideal of a
$*$-ring, it is easy to check that $R/I$ is also a $*$-ring.

\begin{thm}\label{thm2} Let $I$ be a $*$-ideal of a $*$-ring $R$. If $I\subseteq J(R)$,
then $R$ is strongly $J$-$*$-clean if and only if
\begin{enumerate}
\item[\rm{(1)}] $R/I$ is strongly $J$-$*$-clean;
\item[\rm{(2)}] $R$ is abelian;
\item[\rm{(3)}] Every idempotent lifts modulo $I$.
\end{enumerate}
\end{thm}
\begin{proof} Suppose $R$ is strongly $J$-$*$-clean. Then $R$ is
an abelian exchange ring, and so every idempotent lifts modulo
$I$. We attempt to prove that $R/I$ is strongly $J$-$*$-clean, by
showing that $(3)$ of Theorem 3.2. For any $a\in R$, there exist a
projection $e\in R$ and a unit $u\in R$ such that $a=e+u$; hence,
$\overline{a}=\overline{e}+\overline{u}$ in $R/I$. Assume that
there exist a projection $\overline{f}\in R/I$ and a unit
$\overline{v}\in R/I$ such that
$\overline{a}=\overline{f}+\overline{v}$. Then, we can find an
idempotent $g\in R$ such that $f=g+r$ for some $r\in I$. Hence,
$a=g+(v+r+t)$ for some $t\in J(R)$. Obviously,
$g-g^*=f-r-f^*+r^*\in J(R)$. As $ag=ga, ag^*=g^*a$ and $v+r+t\in
U(R)$, it follows by Theorem~\ref{corJ}(5) that $g=e$, and so
$\overline{f}=\overline{e}$ in $R/I$. Therefore $R/I$ is strongly
$J$-$*$-clean.

Conversely, assume that $(1),(2)$ and $(3)$ hold. For any $a\in
R$, it follows from Theorem~\ref{corJ} that there exist a
projection $\overline{e}\in R/I$ and a unit $\overline{u}\in R/I$
such that $\overline{a}=\overline{e}+\overline{u}$. As $e-e^2\in
I$, by hypothesis, there exists an idempotent $f\in R$ such that
$e-f\in I$. Since every unit lifts modulo $I$, we may assume that
$u\in U(R)$. Thus, $a=f+u+r$ for some $r\in I$. Set $v=u+r$. Then
$a=f+v$ with $f=f^2\in R, v\in U(R)$. As $R$ is abelian, $af=fa$
and $af^*=f^*a$. Further, $f-f^*\equiv e-e^*\equiv 0 \, (
\mbox{mod}~ I)$ and so $f-f^* \in J(R)$. Suppose that $a=g+w$ with
$g=g^2\in R, w\in U(R)$, $ag=ga, ag^*=g^*a$ and $g-g^*\in J(R)$.
Then $\overline{a}=\overline{g}+\overline{w}$ in $R/I$. Clearly,
$R/I$ is uniquely clean, and so $f-g\in I\subseteq J(R)$. As
$fg=gf$, we see that $(f-g)^3=(f-2fg+g)(f-g)=f-g$, and so $f=g$.
In light of Theorem~\ref{corJ}(5), $R$ is strongly
$J$-$*$-clean.\end{proof}

Recall that a $*$-ring is {\em $*$-Boolean} in the case that every
element is a projection.

\begin{cor} Let $R$ be a $*$-ring $R$. Then $R$ is strongly $J$-$*$-clean if and only if
\begin{enumerate}
\item[\rm{(1)}] $R/J(R)$ is $*$-Boolean;
\item[\rm{(2)}] $R$ is abelian;
\item[\rm{(3)}] Every idempotent lifts modulo $J(R)$.
\end{enumerate}
\end{cor}
\begin{proof} One easily checks that $J(R)$ is a $*$-ideal of $R$, and thus establishing the claim by Theorem~ \ref{thm2}.\end{proof}

Let $P(R)$ be the {\it prime radical} of $R$, i.e., the
intersection of all prime ideals of $R$. Recall that $a\in R$ is
{\em strongly nilpotent} if for every sequence $a_0,a_1,\cdots
,a_i,\cdots$ such that $a_0 =a$ and $a_{i+1}\in a_iRa_i$, there
exists an $n$ with $a_n=0$. As is well known, the prime radical
$P(R)$ is the set of all strongly nilpotent elements in $R$.

\begin{cor}\label{st1} A $*$-ring $R$ is strongly $J$-$*$-clean if and only
if $R$ is abelian and $R/P(R)$ is strongly $J$-$*$-clean.
\end{cor}
\begin{proof} Let $a\in P(R)$. For every sequence $a_0,a_1,\cdots ,a_i,\cdots$ such
that $a_0 =a^*$ and $a_{i+1}\in a_iRa_i$, we get a sequence
$a_0^*,a_1^*,\cdots ,a_i^*,\cdots$ such that $a_0^*=a$ and
$a_{i+1}^*\in a_i^*Ra_i^*$. As $a\in R$ is strongly nilpotent, we
can find some $n$ such that $a_n^*=0$, and so $a_n=0$. This
implies that $a^*$ is strongly nilpotent; hence, $a^*\in P(R)$. We
infer that $P(R)$ is a $*$-ideal. As every idempotent lifts modulo
$P(R)$, we complete the proof by Theorem~\ref{thm2}.\end{proof}

In \cite{LZ}, Li and Zhou proved that a $*$-ring $R$ is strongly
$*$-clean if and only if $R/J(R)$ is strongly $*$-clean, every
projection is central and every projection lifts to a projection
modulo $J(R)$. Analogous to the previous discussion, we easily
prove that a $*$-ring $R$ is strongly $*$-clean if and only if
$R/J(R)$ is strongly $*$-clean; $R$ is abelian and every
idempotent lifts modulo $J(R)$.

\section{Certain Extensions}

By applying the preceding results, we will construct various
examples of strongly $J$-$*$-clean rings. Let $R$ be a $*$-ring,
and let $R[i]=\{ a+bi~|~a,b\in R, i^2=-1\}$. Then $R[i]$ is also a
$*$-ring by defining $*: a+bi\mapsto a^{*}+b^{*}i$.

\begin{lem}\label{eight} Let $R$ be a ring with $2\in J(R)$. Then $U\big(R[i]\big)=\{ a+bi~|~a+b\in U(R)\}$.
\end{lem}
\begin{proof} Assume that $(a+bi)(c+di)=1$. Then $ac-bd=1$ and
$ad+bc=0$. Thus, $(a+b)(c+d)=ac+bd+ad+bc=(ac-bd)+2bd=1+2bd\in
U(R)$. This implies that $a+b\in R$ is right invertible. As a
result, we show that $U\big(R[i]\big)\subseteq \{ a+bi~|~a+b\in
U(R)\}$.

Assume that $a+b\in U(R)$. Then $a-b=(a+b)-2b\in U(R)$. Clearly,
$a(a-b)^{-1}a=(a-b+b)(a-b)^{-1}(a-b+b)=\big(1+b(a-b)^{-1}\big)(a-b+b)=a+b+b(a-b)^{-1}b$.
Therefore $a(a-b)^{-1}a(a+b)^{-1}-b(a-b)^{-1}b(a+b)^{-1}=1$.
Likewise, $b(a-b)^{-1}a-a(a-b)^{-1}b=0$. It is easy to check that
$$\begin{array}{lll}
&&(a+bi)(a-b)^{-1}(a-bi)(a+b)^{-1}\\
&=&\big(a(a-b)^{-1}a(a+b)^{-1}+b(a-b)^{-1}b(a+b)^{-1}\big)
 +\big(b(a-b)^{-1}a-a(a-b)^{-1}b\big)i\\
 &=&1+2b(a-b)^{-1}b(a+b)^{-1}\\
 &\in &U(R).
 \end{array}$$
 Thus, $a+bi\in R[i]$ is right invertible.
 Analogously, $(a-b)^{-1}(a-bi)(a+b)^{-1}(a+bi)\in U(R)$.
 Therefore $a+bi\in U\big(R[i]\big)$, as required.\end{proof}

\begin{prop}\label{pac7} Let $R$ be a $*$-ring. Then $R[i]$ is strongly $J$-$*$-clean if and only if so is $R$.
\end{prop}
\begin{proof} Suppose that $R[i]$ is strongly $J$-$*$-clean. Then
$2\in J(R)$. Further, every idempotent in $R[i]$ is a projection,
and it is central. Let $a\in R$. Then we can find a projection
$e+fi\in R[i]$ and an element $u+vi\in J\big(R[i]\big)$ such that
$a=(e+fi)+(u+vi)$ and $a(e+fi)=(e+fi)a$. Thus, $a=e+u$ and
$ae=ea$. As $e+fi\in R[i]$ is central, $e\in R$ is central. Since
$(e+fi)^*=e+fi$, we see that $e^*=e$. From $e+fi=(e+fi)^2$, we get
$e^2-f^2=e$ and $2ef=f$. This implies that $(2e-1)f=0$, and then
$f=0$. Hence, $e\in R$ is a projection. It is easy to verify that
$u\in J(R)$, and therefore $R$ is strongly $J$-$*$-clean.

Conversely, assume that $R$ is strongly $J$-$*$-clean. Then $R$ is
an abelian exchange ring. In addition, every idempotent in $R$ is
a projection and $2\in J(R)$. Let $a+bi\in R[i]$. By hypothesis,
there exist projections $e,f\in R$ and $u,v\in J(R)$ such that
$a=e+u,b=f+v, ae=ea, bf=fb$. Thus,
$a+bi=(e+f)+\big(u-f+(f+v)i\big)$. Clearly, $(e+f)^2-(e+f)=2ef\in
J(R)$. As every idempotent lifts modulo $J(R)$, we can find an
idempotent $g\in R$ such that $e+f=g+r$ where $r\in J(R)$. Thus,
$a+bi=g+\big(r+u-f+(f+v)i\big)$ where $g=g^2=g^*$ and
$(a+bi)g=g(a+bi)$.

Write $x=r+u-f$ and $y=f+v$. Then $x+y=r+u+v\in J(R)$. For any
$c+di\in R[i]$, we see that $1-(x+yi)(c+di)=(1-xc+yd)-(xd+yc)i$.
As $(1-xc+yd)-(xd+yc)=(1-xc-yd)-(xd+yc)+2yd=1-(x+y)(c+d)+2yd\in
U(R)$. In light of Lemma~\ref{eight}, $1-(x+yi)(c+di)\in
U\big(R[i]\big)$, and so $x+yi\in J\big(R[i]\big)$. Therefore
$R[i]$ is strongly $J$-$*$-clean, as asserted.\end{proof}

Let $R$ be a $*$-ring. Then $*$ induces an involution of the power
series ring $R[[x]]$, denoted by $*$, where
$\big(\sum\limits_{i=0}^{\infty}a_ix^i\big)^*
=\sum\limits_{i=0}^{\infty}a^*_ix^i$. This induces the involution
on $R[[x]]/(x^{n})$ $(n\geq 1)$.

\begin{prop}\label{pac11}  Let $R$ be a $*$-ring. Then the following are equivalent:
\begin{itemize}
\item[{\rm(1)}] $R$ is strongly $J$-$*$-clean.
\item[{\rm(2)}] $R[[x]]$ is strongly $J$-$*$-clean.
\item[{\rm(3)}] $R[[x]]/(x^n)$ is strongly $J$-$*$-clean for all $n\geq 2$.
\item[{\rm(4)}] $R[[x]]/(x^2)$ is strongly $J$-$*$-clean.
\end{itemize}
\end{prop}
\begin{proof} $(1)\Rightarrow (2)$ Since $R$ is strongly $J$-$*$-clean,
$R[[x]]$ is strongly $*$-clean by \cite[Corollary 2.10]{LZ}.
$R[[x]]$ is also strongly $J$-clean by \cite[Example 16.4.17]{Ch}.
Hence $R[[x]]$ is strongly $J$-$*$-clean by
Proposition~\ref{prop1}.

$(2)\Rightarrow (3)$ In view of Lemma ~\ref{str3}, $R[[x]]$ is
uniquely clean, and then so is $S:=R[[x]]/(x^n)$ by \cite[Theorem
22]{NZ}. For any $f\in S$, it follows from Lemma ~\ref{str3} that
$f(0)-\big(f(0)\big)^*\in J(R)$, and so
$f-f^*=f(0)-(f(0))^*+\sum\limits_{i=1}^{n-1}b_ix^i\in J(S)$. By
using Lemma ~\ref{str3} again, $R[[x]]/(x^n)$ is a strongly
$J$-$*$-clean ring.

$(3)\Rightarrow (4)$ is trivial.

$(4)\Rightarrow (1)$ Let $S=R[[x]]/(x^2)$. For any $a=a+0x\in S$, there
exists a projection $e+fx\in S$ such that
$a(e+fx)=(e+fx)a$ and $a-(e+fx)\in J(S)$. This
implies that $e\in R$ is a projection, $ae=ea$ and $a-e\in J(R)$.
Therefore $R$ is a strongly $J$-$*$-clean ring.\end{proof}

Since $R[x]/(x^2)$ and $R[[x]]/(x^2)$ are isomorphic,
Proposition~\ref{pac11} also implies that a $*$-ring is strongly
$J$-$*$-clean if and only if $R[x]/(x^2)$ is.

Let $R$ be a $*$-ring and $G$ be a group. Then $*$ induces an
involution of the group ring $RG$, denoted by $*$, where $(\sum_g
a_g g)^*=\sum_g a_g^* g^{-1}$ (see \cite[Lemma 2.12]{LZ}). A group
$G$ is called {\em locally finite} if every finitely generated
subgroup of $G$ is finite. A group $G$ is a {\em $2$-group} if the
order of each element of $G$ is a power of $2$.

\begin{prop}\label{three} Let $R$ be a $*$-ring, and let $G$ be a locally finite group. Then $RG$ is
strongly $J$-$*$-clean if and only if $R$ is strongly
$J$-$*$-clean and $G$ is a $2$-group.
\end{prop}
\begin{proof} Suppose that $RG$ is  strongly $J$-$*$-clean. Then it is uniquely clean.
By virtue of \cite[Theorem 12]{CHZ}, $R$ is uniquely clean and $G$
is a $2$-group. In view of Lemma ~\ref{str3}, $2\in J(RG)$, and so
$2\in J(R)$. In addition, every idempotent in $RG$ is central, and
so every idempotent in $R$ is central in $R$. According to
\cite[Lemma 11]{CHZ}, every idempotent in $RG$ is in $R$.
Therefore every idempotent in $R$ is a projection. According to
Theorem~\ref{corJ}, $R$ is strongly $J$-$*$-clean.

Conversely, assume that $R$ is a strongly $J$-$*$-clean ring and
$G$ is a $2$-group. By virtue of \cite[Theorem 13]{CHZ}, $RG$ is
uniquely clean. As $R$ is strongly $J$-$*$-clean, we see that
$2\in J(R)$. Further, every idempotent in $R$ is central.
According to \cite[Lemma 11]{CHZ}, every idempotent in $RG$ is in
$R$. Therefore every idempotent in $RG$ is a projection. By using
Proposition~\ref{prop1}, $RG$ is strongly $J$-$*$-clean.
\end{proof}

\begin{cor}\label{four} Let $R$ be a $*$-ring, and let $G$ be a solvable group. Then $RG$ is strongly $J$-$*$-clean
if and only if $R$ is strongly $J$-$*$-clean and $G$ is a
$2$-group.
\end{cor}
\begin{proof} The proof of necessity is the same as in Proposition~\ref{three}.
Conversely, assume that $R$ is a strongly $J$-$*$-clean ring and
$G$ is a $2$-group. Analogously to the consideration in
\cite[Theorem 13]{CHZ}, $G$ is locally finite, and then the result
follows from Proposition~\ref{three}.
\end{proof}

\vskip8mm \hspace{-1.5em}{\Large\bf Acknowledgements}\\

\hspace{-1.5em}The authors are grateful to the referee for his/her
suggestions which corrected many errors in the first version and
made the new one clearer.

\end{document}